\documentclass[12pt,twoside,reqno]{amsart}

\usepackage{amsmath,amssymb}
\usepackage{amsthm}
\usepackage{ascmac,color,xcolor,braket}
\usepackage[dvipdfmx]{graphicx}
\usepackage[abbrev]{amsrefs}
\usepackage{latexsym}
\newcommand{\aut}{\mathop{\rm Aut}\nolimits}
\newcommand{\setrel}{\mathrel{}\middle|\mathrel{}}

\newtheorem{thm}{Theorem}[section]

\newtheorem{lem}[thm]{Lemma}
\newtheorem{cor}[thm]{Corollary}
\newtheorem{remark}[thm]{Remark}

\theoremstyle{definition} 

\newtheorem*{ack*}{Acknowledgment}

\makeatletter

\@addtoreset{equation}{section}
\makeatother

\begin{document}

\title{
An upper bound for higher order eigenvalues of symmetric graphs
}
\author{Shinichiro Kobayashi}
\address{Mathematical Institute, Tohoku University, Sendai 980-8578, Japan}
\email{shin-ichiro.kobayashi.p3@dc.tohoku.ac.jp}
\date{June 13, 2020.} 
\maketitle

\begin{abstract}
In this paper, we derive an upper bound for higher order eigenvalues of the normalized Laplace operator associated with a symmetric finite graph in terms of lower order eigenvalues. 
\end{abstract}

\section{Introduction}
Let $G$ be a connected, finite, simple and undirected graph of $N$ vertices.
Let $\Delta$ be the normalized Laplace operator assiciated with $G$. The operator $-\Delta$ is identified with a non-negative definite real symmetric matrix of size $N$.
Denote by $\lambda_0\leq\lambda_1\leq \dots\leq\lambda_{N-1}$ all eigenvalues of $\Delta$ counted with multiplicity.
For any connected graph, we have $\lambda_0=0$ and its multiplicity is $1$. All the eigenvalues lie in the interval $[0,2]$.
We consider the following question:
Are there other contraints on the spectrum $\{\lambda_i\}_{i=0}^{N-1}$?
In particular, is $\lambda_{k+1}$ controlled by past eigenvalues, $\lambda_1,\dots,\lambda_{k}$?
This question is a discrete analogue of the so-called Payne-P\'{o}lya-Weinberger's inequality.
For the Dirichlet eigenvalues $0<\lambda_1<\lambda_2\leq \lambda_3\leq \cdots \uparrow \infty$ of the Laplacian on a bounded domain in the Euclid plane, Payne-P\'{o}lya-Weinberger \cite{MR73046, MR84696} proved that
\[
\lambda_{k+1}-\lambda_{k}\leq \frac2k\sum_{i=1}^k\lambda_i.
\]
This result is extended to arbitrary dimension by Thompson \cite{MR257592}.
Later, Hile and Protter \cite{MR578204} and Yang \cite{Yang} proved sharper inequalities.
In particular, Yang \cite{Yang} proved that
\begin{equation}\label{Yangineq}
\sum_{i=1}^k(\lambda_{k+1}-\lambda_i)^2\leq \frac4n\sum_{i=1}^k(\lambda_{k+1}-\lambda_i)\lambda_i.
\end{equation}
Chung and Oden proposed to study of the discrete analogue of their results. 
For the Dirichlet eigenvalues $\{\lambda_i\}_{i\geq 1}$ of the normalized Laplacian on a connected finite subgraph in the integer lattice of rank $n$, Hua, Lin and Su \cite{HLS} proved that
\[
\sum_{i=1}^k(\lambda_{k+1}-\lambda_i)^2(1-\lambda_i)\leq \frac4n\sum_{i=1}^k(\lambda_{k+1}-\lambda_i)\lambda_i.
\]

How about the case of the Laplacian without boundary conditions?
Unlike the case of the Dirichlet boundary condition, $0$ is always an eigenvalue.
For the eigenvalues $\{\lambda_i\}_{i\geq 0}$ with $\lambda_0:=0$ of the Laplacian on a compact Riemannian homogeneous manifold, Cheng and Yang \cite{MR2115463} proved that
\begin{equation}\label{thm: CY}
\sum_{i=0}^k(\lambda_{k+1}-\lambda_i)^2\leq \sum_{i=1}^k(\lambda_{k+1}-\lambda_i)(4\lambda_i+\lambda_1).
\end{equation}
In this paper, we consider a discrete analogue of \eqref{thm: CY}.
More precisely, for a finite symmetric graph, we prove a discrete analogue of \eqref{thm: CY}.
\begin{thm}\label{main thm}
Let $G$ be an symmetric finite graph with $N$ vertices.
Denote by $0=\lambda_0<\lambda_1\leq \lambda_2\leq\cdots\leq\lambda_{N-1}$ the all eigenvalues of the normalized Laplace operator. Then, for any non-zero eigenvalue $\lambda$ of $\Delta$, we have
\[
\sum_{i=0}^k(\lambda_{k+1}-\lambda_i)^2(1-\lambda_i)\leq
\sum_{i=0}^k(\lambda_{k+1}-\lambda_i)(2(2-\lambda)\lambda_i+\lambda).
\]
\end{thm}
By using Chebyshev's sum inequality, we obtain an upper bound of $\lambda_{k+1}$ in terms of $\lambda_1,\dots,\lambda_k$.
\begin{thm}\label{main thm 2}
In the same setting as Theorem \ref{main thm}, we have
\[
\lambda_{k+1}\leq \frac{(k+1)\lambda_1 + \sum_{i=1}^k((5-2\lambda_1)\lambda_i-\lambda_i^2)}{\sum_{i=0}^k(1-\lambda_i)}.
\]
\end{thm}
Let $\mu_1:=\lambda_1$ and $m$ be the multiplicity of $\mu_1$. If $G$ is not a complete graph, then we can consider $\mu_2:=\lambda_{m+1}$, i.e., the second smallest positive eigenvalue.
We have a upper bound for the ratio $\mu_2/\mu_1$ in terms of the multiplicity of $\mu_1$.
\begin{cor}\label{main cor}
In the same setting as Theorem \ref{main thm}, let $m$ be the multiplicity of $\mu_1$ and put $\mu_2:=\lambda_{m+1}$. Then, we have
\[
\frac{\mu_2}{\mu_1}\leq 3m+1.
\]
\end{cor}

\begin{ack*}
The author would like to thank Professor Takashi Shioya for helpful comments.
\end{ack*}

\section{Preliminaries}
In this section, unless otherwise stated, we assume that all graphs are connected, finite, simple and undirected.
We recall some basic facts on the theory of eigenvalues of a regular graph.
Let $G=(V,E)$ be a $d$-regular graph, $d\geq 1$, and put $N:=\#V$.
If two vertices $x,y\in V$ are adjacent, then we denote this situation by $x\sim y$. Note that since $G$ is undirected, $x\sim y$ if and only if $y\sim x$. 
The \emph{normalized Laplace operator} $\Delta$ acting on the space $C(V)$ of functions on $V$ is defined by 
\[
\Delta u(x):=\frac{1}{d}\sum_{y\sim x}\left(u(y)-u(x)\right),\, u\in C(V), x\in V.
\]
The normalized Laplace operator is identified with the real-symmetric matrix $D^{-1}A-I$, where $D$ is the scalar matrix with diagonal entries $d$, $A$ is the adjacency matrix of $G$ and $I$ is the identity matrix.
A complex number $\lambda$ is called an \emph{eigenvalue} of $\Delta$ if there exists $u\in C(V)\setminus \{0\}$ such that $\Delta u+\lambda u=0$ holds. In this case, the function $u$ is called an \emph{eigenfunction} with eigenvalue $\lambda$. For an eigenvalue $\lambda$ of $\Delta$, we denote by $W_{\lambda}$ the space of all functions $u\in C(V)$
satisfying $\Delta u+\lambda u=0$ and we call the dimension of $W_{\lambda}$ \emph{multiplicity} of $\lambda$.
Let us denote the eigenvalues of $\Delta$ by $\lambda_0\leq\lambda_1\leq \lambda_2\leq \dots\leq \lambda_{N-1}$, counted with multiplicity.
We define a inner product $\langle\cdot,\cdot\rangle$ on $C(V)$ by
\[
\langle u,v\rangle:=\sum_{x\in V}u(x)v(x)d.
\]
We denote by $\|\cdot\|$ the norm induced by the inner product $\langle \cdot,\cdot\rangle$.
We list up some elementary facts on eigenvalues and eigenfunctions without proofs.
\begin{itemize}
\item $0$ is an eigenvalue of multiplicity $1$ and constant functions are eigenfunctions with eigenvalue $0$.
\item All eigenvalues lie in the interval $[0,2]\subset \mathbb{R}$.
\item There exists a orthonormal basis $\{u_{i}\}_{i=0}^{N-1}$ of $C(V)$ such that each function $u_{i}$ is an eigenfunction with eigenvalue $\lambda_i.$
\end{itemize}
By the min-max formula, each eigenvalue $\lambda_k$ has a variational characterization:
\[
\lambda_k = \inf\left\{\frac{\sum_{x\sim y}(u(y)-u(x))^2}{2d\sum_{V}u^2}\setrel u\neq 0, \langle u,u_i\rangle=0,\, i=0,\dots, k-1\right\},
\]
where the symbol $\sum_{x\sim y}$ means the summation over all unordered pairs $(x,y)$ such that $x\sim y$.
In particular, we have
\begin{equation}\label{spectral gap}
\lambda_1 = \inf\left\{\frac{\sum_{x\sim y}(u(y)-u(x))^2}{2d\sum_{V}u^2}\setrel u\neq 0, \sum_Vu=0 \right\}.
\end{equation}
We shall derive a general upper bound for $\lambda_1$.
\begin{lem}\label{lem: upperbound}
For any regular graph $G$ but a complete graph, we have
\[
\lambda_1\leq 1.
\]
\end{lem}
\begin{proof}
Since $G$ is not complete, there exist two vertices $x_0,y_0\in V$ such that $x_0\not \sim y_0$. We define a function $u\in C(V)$ by
\[
u(x):=
\begin{cases}
1 & \text{if}\ x=x_0, \\
-1 & \text{if}\ x=y_0, \\
0 & \text{otherwise}.
\end{cases}
\]
Clearly, the function $u$ satisfies $\sum_Vu=0$. From \eqref{spectral gap}, we have
\[
\lambda_1\leq \frac{\sum_{x\sim y}(u(y)-u(x))^2}{2d\sum_{V}u^2}=1.
\]
\end{proof}
\begin{remark}
If $G$ is the complete graph of degree $d$, then $\lambda_1=1+1/d$.
\end{remark}
Let $\Gamma\colon C(V)\times C(V)\to C(V)$ be the \emph{carr\'{e} du champ operator} associated to $\Delta$, i.e., for $u,v\in C(V)$,
\[
\Gamma(u,v):=\frac12\left(\Delta(uv)-(\Delta u)v-u\Delta v\right).
\]
For two vertices $x,y\in V$ with $x\sim y$, we define the \emph{difference operator} $\nabla_{xy}\colon C(V)\to C(V)$
by
\[
\nabla_{xy}u:=u(y)-u(x),\, u\in C(V).
\]
By a simple calculation, we have
\[
\Gamma(u,v)(x)=\frac{1}{2d}\sum_{y\sim x}(\nabla_{xy}u)(\nabla_{xy}v),\, x\in V.
\]
The carr\'{e} du champ $\Gamma(u,v)$ is an analogy of $\langle \nabla u,\nabla v\rangle$ in the context of Riemannian geometry, where $\nabla$ is the gradient operator.
We list up some identities for $\Gamma$.
\begin{lem}\label{lem: calculus}
Let $u,v, v_1,v_2\in C(V)$.
\begin{enumerate}
\item $\langle u, \Delta v\rangle = -\sum_{V}\Gamma(u,v)d.$
\item For any $x\in V$, we have
\begin{align*}
\Gamma(u,v_1v_2)(x) = \Gamma(u, v_1)&v_2(x) + \Gamma(u, v_2)v_1(x) \\
&+\frac{1}{2d}\sum_{y\sim x}(\nabla_{xy}u)(\nabla_{xy} v_1)(\nabla_{xy}v_2).
\end{align*}
In particular,
\[
\sum_{V}\Gamma(u, v_1v_2) =\sum_{V} (\Gamma(u, v_1)v_2 + \Gamma(u, v_2)v_1).
\]
\end{enumerate}
\end{lem}
Making use of the min-max formula and appropriate trial functions, we have the following lemma.
\begin{lem}\label{CYineq}
Let $k\geq 1$ be an integer.
For any function $h\in C(V)$, we have
\[
\frac12\sum_{i=0}^{k}(\lambda_{k+1}-\lambda_i)^2\Phi_i(h)\leq \sum_{i=0}^{k}(\lambda_{k+1}-\lambda_i)\|2\Gamma(h,u_i) + u_i\Delta h\|^2,
\]
where $\Phi_i(h)=\sum_{x\sim y}u_i(x)u_i(y)(\nabla_{xy}h)^2$.
\end{lem}
\begin{proof}
Let $h\in C(V)$. For $i=0,\dots,k$, define $\varphi_i\in C(V)$ as the orthogonal projection of $hu_i$ to the subspace spanned by $\{u_{k+1},\dots,u_{N-1}\}$, i.e., 
\[
\varphi_i:=hu_i-\sum_{j=0}^ka_{ij}u_j, \, a_{ij}:=\langle hu_i,u_j\rangle.
\]
Clearly the function $\varphi_i$ is perpendicular to $u_0,\dots,u_{k}$. The min-max formula yields
\begin{equation}\label{minmax}
\lambda_{k+1}\|\varphi_i\|^2 \leq \frac12\sum_{x\sim y}(\nabla_{xy}\varphi_i)^2 = \sum_{V}\Gamma(\varphi_i,\varphi_i)d.
\end{equation}
From (1) in Lemma \ref{lem: calculus} and the fact that $\langle\varphi_i,u_j\rangle=0$ for $j=0,\dots,k$, we have
\begin{align*} 
\sum_V\Gamma(\varphi_i,\varphi_i)d&= -\langle \varphi_i,\Delta \varphi_i\rangle \\
&=-\langle\varphi_i, 2\Gamma(h,u_i)+u_i\Delta h-\lambda_iu_ih+\sum_{j=0}^ka_{ij}\lambda_ju_j\rangle \\
&=-\langle\varphi_i, 2\Gamma(h,u_i)+u_i\Delta h-\lambda_iu_ih\rangle\\
&= -\langle\varphi_i,2\Gamma(h,u_i)+u_i\Delta h\rangle + \lambda_i \|\varphi_i\|^2.
\end{align*}
From \eqref{minmax}, we obtain
\begin{equation}\label{first}
(\lambda_{k+1}-\lambda_i)\|\varphi_i\|^2 \leq -\langle\varphi_i,2\Gamma(h,u_i)+u_i\Delta h\rangle.
\end{equation}
Let $A_i$ be the right hand side of (\ref{first}).
We estimate $A_i$ in two ways.
First, we claim that
\begin{equation}\label{second}
A_i = \frac12\sum_{x\sim y}u_i(x)u_i(y)(\nabla_{xy}h)^2 + \sum_{j=0}^{k} (\lambda_i-\lambda_j)a_{ij}^2.
\end{equation}
To see \eqref{second}, we use Lemma \ref{lem: calculus}. By the definition of $\varphi_i$,
\begin{align*}
A_i = \sum_{j=0}^{k}a_{ij}\langle u_j, u_i\Delta h +2\Gamma(h, u_i)\rangle - d\sum_{V}(hu_i^2\Delta h +2hu_i\Gamma(h,u_i)).
\end{align*}
The first term is equal to $\sum_{j=0}^k(\lambda_i-\lambda_j)a_{ij}^2$. Indeed, by the definition of $\Gamma(h,u_i)$ and Lemma \ref{lem: calculus}, we have
\begin{align}\label{bij}
\nonumber \langle u_j, u_i\Delta h +2\Gamma(h, u_i)\rangle &= \langle u_j, \Delta(hu_i) + \lambda_ihu_i \rangle\\
&= \nonumber \lambda_ia_{ij} - \langle \lambda_ju_j,hu_i\rangle \\
&= (\lambda_i-\lambda_j)a_{ij}.
\end{align}
The second term is equal to $\sum_{x\sim y}u_i(x)u_i(y)(\nabla_{xy}h)^2/2$. Indeed,
\begin{align*}
-\langle hu_i^2, \Delta h\rangle &= \sum_V\Gamma(hu_i^2, h)d \\
&= \sum_{V}(h\Gamma(u_i^2,h) + u_i^2 \Gamma(h,h))d \\
&= \frac{1}{2}\sum_{x\sim y}\left((\nabla_{xy}u_i)^2h(x)(\nabla_{xy}h)+u_i(x)^2(\nabla_{xy}h)^2\right) \\
&+ \sum_{V}2hu_i\Gamma(h,u_i)d \\
&= \frac{1}{2}\sum_{x\sim y}u_i(x)u_i(y)(\nabla_{xy}h)^2+\sum_{V}2hu_i\Gamma(h,u_i)d.
\end{align*}

Second, we claim that 
\begin{equation}\label{third}
(\lambda_{k+1}-\lambda_i)A_i\leq \|u_i\Delta h + 2\Gamma(u_i,h)\|^2 - \sum_{j=0}^{k}(\lambda_i-\lambda_j)^2a_{ij}^2.
\end{equation}
From the definition of $A_i$, we have
\[
A_i=-\langle \varphi_i, 2\Gamma(h,u_i)+u_i\Delta h -\sum_{j=0}^k(\lambda_i-\lambda_j)a_{ij}u_j\rangle.
\]
Applying the Cauchy-Schwartz inequality to the definition of $A_i$ and taking account into \eqref{first} and \eqref{bij}, we have
\[
(\lambda_{k+1}-\lambda_i)A_i^2 \leq A_i(\|2\Gamma(h,u_i)+u_i\Delta h\|^2-\sum_{j=0}^k(\lambda_i-\lambda_j)^2a_{ij}^2).
\]
From \eqref{second} and \eqref{third}, we obtain
\begin{align*}
&\frac12\sum_{i=0}^k(\lambda_{k+1}-\lambda_i)^2\sum_{x\sim y}u_i(x)u_i(y)(\nabla_{xy}h)^2 + \sum_{i,j=0}^k(\lambda_{k+1}-\lambda_i)^2(\lambda_i-\lambda_j)a_{ij}^2\\
&\leq \sum_{i=0}^k(\lambda_{k+1}-\lambda_i)\|2\Gamma(h,u_i)+u_i\Delta h\|^2-\sum_{i,j=0}^k(\lambda_{k+1}-\lambda_i)(\lambda_i-\lambda_j)^2a_{ij}^2.
\end{align*}
Since $\sum_{i,j=0}^k(\lambda_{k+1}-\lambda_i)^2(\lambda_i-\lambda_j)a_{ij}^2=-\sum_{i,j=0}^k(\lambda_{k+1}-\lambda_i)(\lambda_i-\lambda_j)^2a_{ij}^2$, we complete the proof.
\end{proof}

\section{proof of main theorem}
In this section, we give a proof of Theorem 1.1. In order to complete the proof, we use some symmetries of eigenfunctions on a symmetric graph.
\subsection{Symmetries of eigenfunctions on a symmetric graph}
We derive some properties of  eigenfunctions on a symmetric graph.
In particular, Lemma \ref{lem: const2} is peculiar to symmetric graphs.
A graph $G=(V,E)$ is said to be \emph{symmetric} if for any two edges $(x,y), (x',y')\in E$, there exists an automorphism $\gamma$ of $G$ such that $x'=\gamma x$ and $y'=\gamma y$ hold.
We denote by $\aut (G)$ the group of automorphisms of $G$.
Note that symmetric graphs are vertex-transitive, i.e., $\aut (G)$ acts transitively on $V$, and thus regular.
We say that a vector subspace $W$ of $C(V)$ is \emph{invariant} if for any $u\in W$ and $\gamma\in \aut (G)$, $\gamma u\in W$, where $\gamma u$ is defined by $\gamma u(x):=u(\gamma x), x\in V$. 
\begin{lem}\label{lem: const1}
Let $G=(V,E)$ be a vertex-transitive graph.
Let $W$ be an invariant vector subspace of $C(V)$ of dimension $m$ and let $\{ u_{\alpha} \}_{\alpha=1}^{m}$
be an orthonormal basis of $W$. Then, the function $\lvert u_1\rvert^2+\dots+\lvert u_m\rvert^2$ is constant
and its value is $m/d\#V$.
\end{lem}

\begin{proof}
Put $f(x):=\lvert u_1(x)\rvert^2+\dots+\lvert u_m(x)\rvert^2$. By the invariance of $W$, the family $\{\gamma u_{\alpha}\}_{\alpha=1}^m$ is also an orthonormal basis of $W$ for any $\gamma\in \aut (G)$.
For fixed $x\in V$, it is easy to see that the sum $\lvert u_1(x)\rvert^2+\dots+\lvert u_m(x)\rvert^2$ is independent of the choice of an orthonormal basis $\{u_{\alpha}\}$. Thus,
\[
f(\gamma x)=\sum_{\alpha=1}^m\lvert \gamma u_{\alpha}(x)\rvert^2 = \sum_{\alpha=1}^m\lvert u_{\alpha}(x)\rvert^2=f(x).
\]
The transitivity of the action of $\aut(G)$ yields that $f$ is constant.
Let $C$ be the value of $\lvert u_1(x)\rvert^2+\dots+\lvert u_m(x)\rvert^2$.
By multiplying $d$ and summing over $x\in V$, we have
\[
Cd\#V=\sum_{\alpha=1}^m\sum_{x\in V}\lvert u_{\alpha}(x)\rvert^2d=m.
\]
\end{proof}

\begin{lem}\label{lem: const2}
Let $G$ be a symmetric graph.
Let $\lambda$ be an eigenvalue of $\Delta$ and let $\{u_{\alpha}\}_{\alpha=1}^m$ be an orthonormal basis of $W_{\lambda}$. Then, the function $g(x,y):=\sum_{\alpha=1}^m\lvert\nabla_{xy}u_{\alpha}\rvert^2, x\sim y$, is constant
and its value is $m\lambda/\#E$.
\end{lem}

\begin{proof}
Since $W_{\lambda}$ is an invariant vector subspace of $C(V)$, the family $\{\gamma u_{\alpha}\}_{\alpha=1}^m$ is also an orthonormal basis of $W_{\lambda}$ for any $\gamma\in \aut(G)$.
Since the sum $\sum_{\alpha=1}^m\lvert\nabla_{xy}u_{\alpha}\rvert^2$ is independent of the choice of an orthonormal basis $\{u_{\alpha}\}$, we have
\[
g(\gamma x,\gamma y)=\sum_{\alpha=1}^m\lvert\nabla_{xy}(\gamma u_{\alpha})\rvert^2=\sum_{\alpha=1}^m\lvert\nabla_{xy}u_{\alpha}\rvert^2 = g(x,y).
\]
The symmetry of $G$ yields that $g$ is constant. Let $C'$ be the value of $g$. By summing over $x\sim y$, we have
\[
2C'\#E=\sum_{\alpha=1}^m\sum_{x\sim y}\lvert \nabla_{xy}u_{\alpha}\rvert^2 = 2\lambda m.
\]
\end{proof}

\begin{cor}\label{cor: const3}
Let $G$ be a symmetric graph.
Let $\lambda$ and $\{u_{\alpha}\}$ be as in Lemma \ref{lem: const2}.
Then, the function $f_3(x,y)=\sum_{\alpha=1}^mu_{\alpha}(x)\nabla_{xy}u_{\alpha}$ is constant and its value is
$-\lambda m/2\#E$.
\end{cor}

\begin{proof}
The constancy of $f_3$ immediately follows from Lemma \ref{lem: const1} and Lemma \ref{lem: const2}.
Let $C_3$ be the value of $f_3$. By summing over $x\sim y$, we have
\[
2C_3\#E=\sum_{\alpha=1}^m\sum_{x\sim y}u_{\alpha}(x)\nabla_{xy}u_{\alpha}.
\]
By interchanging $x$ and $y$,
\[
\sum_{x\sim y}u_{\alpha}(x)\nabla_{xy}u_{\alpha} = -\sum_{x\sim y} u_{\alpha}(y)\nabla_{xy}u_{\alpha}.
\]
Thus, we obtain
\[
2C_3\#E = \frac12 \sum_{\alpha=1}^m\sum_{x\sim y}(u_{\alpha}(x)-u_{\alpha}(y))\nabla_{xy}u_{\alpha} =-\lambda m.
\]
\end{proof}
\subsection{Proof of main theorem}
We prove Theorem \ref{main thm}, Theorem \ref{main thm 2} and Corollary \ref{main cor}.
First, we prove Theorem \ref{main thm}.
\begin{proof}[Proof of Theorem 1.1]
Let $\{u_{\alpha}\}_{\alpha}$ be an orthonormal basis of $E_{\mu}$. Then, we have
\begin{align}\label{term: lhs}
\nonumber \sum_{\alpha=1}^m\sum_{x\sim y}u_i(x)u_i(y)\lvert\nabla_{xy}u_{\alpha}\rvert^2 &=
\frac{\lambda m}{\#E}\sum_{x\sim y}u_i(x)u_i(y) \\
\nonumber&=\frac{\lambda m}{\#E}\sum_{x\in V}u_{i}(x)d\cdot \frac1d\sum_{y\sim x}u_i(y) \\
\nonumber&=\frac{\lambda m}{\#E}(1-\lambda_i)\sum_{x\in V}u_i(x)^2d \\
&=\frac{\lambda m}{\#E}(1-\lambda_i).
\end{align}
Next, we evaluate $\sum_{\alpha}\|2\Gamma (u_{\alpha},  u_i) + u_i\Delta u_{\alpha}\|^2$.
By Jensen's inequality, we have
\[
4\Gamma(u_i,u_{\alpha})(x)^2 = \left(\frac{1}{d}\sum_{y\sim x}(\nabla_{xy}u_i)(\nabla_{xy}u_{\alpha})\right)^2
\leq \frac{1}{d}\sum_{y\sim x}(\nabla_{xy}u_i)^2(\nabla_{xy}u_{\alpha})^2,
\]
which yields
\begin{equation}\label{term: rhs1}
4\sum_{\alpha=1}^m\sum_{x\in V}\Gamma(u_i,u_{\alpha})(x)^2d \leq \frac{2\lambda\lambda_i m}{\#E}.
\end{equation}
By Lemma \ref{lem: const1}, we have
\begin{equation}\label{term: rhs2}
\sum_{\alpha=1}^m\sum_{x\in V}(u_i(x)\Delta u_{\alpha}(x))^2d=\frac{\lambda^2 m}{2\#E}.
\end{equation}
By Lemma \ref{cor: const3},
\begin{align}\label{term: rhs3}
\nonumber -4\lambda\sum_{\alpha=1}^m\sum_{x\in V}u_i(x)u_{\alpha}(x)\Gamma(u_{\alpha},u_i)(x)d &=
\frac{\lambda^2 m}{\#E}\sum_{x\in V}u_i(x)\sum_{y\sim x}\nabla_{xy}u_i \\ 
&= -\frac{\lambda^2\lambda_im}{\#E}.
\end{align}
By letting $h=u_{\alpha}$ in Lemma \ref{CYineq}, summing over $\alpha=1,\dots,m$ and taking account into \eqref{term: lhs}, \eqref{term: rhs1}, \eqref{term: rhs2} and \eqref{term: rhs3}, we obtain
\[
\sum_{i=0}^k(\lambda_{k+1}-\lambda_i)^2(1-\lambda_i)\leq
\sum_{i=0}^k(\lambda_{k+1}-\lambda_i)(2(2-\lambda)\lambda_i+\lambda).
\]
\end{proof}
In order to prove Theorem \ref{main thm 2}, we need some lemmas.
\begin{lem}[Chebyshev's sum inequality]\label{Chebyineq}
Let $N\geq 1$ be an integer and $\{a_i\}_{i=1}^{N}, \{b_i\}_{i=1}^{N}$ two sequences of real numbers.
If both of $\{a_i\}_{i=1}^{N}, \{b_i\}_{i=1}^{N}$ are non-increasing, then
\[
\frac1N\sum_{i=1}^Na_ib_i \geq \left(\frac1N\sum_{i=1}^Na_i\right)\left(\frac1N\sum_{i=1}^Nb_i\right).
\]
\end{lem}

\begin{lem}\label{lem: admatrix}
For any $0\leq k\leq N-1$,
\[
\sum_{i=0}^k(1-\lambda_i)\geq 0
\]
and the equality holds if and only if $k=N-1$.
\end{lem}

\begin{proof}
Let $A$ be the adjacency matrix of $G$ and $\nu_0\geq \nu_1\geq \dots\geq \nu_{N-1}$ be all eigenvalues of $A$.
Since any diagonal entry of $A$ is $0$, $\sum_{i=0}^{N-1}\nu_i$ is also $0$ and $\sum_{i=0}^k\nu_i\geq 0$ for any $k$,
with the equality holds if and only if $k=N-1$.
By the relation between $\Delta$ and $A$, we have
\[
\sum_{i=0}^k(1-\lambda_i) = \frac1d\sum_{i=0}^k\nu_i \geq 0
\]
and equality holds if and only if $k=N-1$.
\end{proof}
Next, we prove Theorem \ref{main thm 2} and Corollary \ref{main cor}.
\begin{proof}[Proof of Theorem 1.2]
By letting $\lambda=\lambda_1$ in Theorem \ref{main thm}, we have
\[
\sum_{i=0}^k(\lambda_{k+1}-\lambda_i)(\lambda_i^2-(\lambda_{k+1}-2\lambda_1+5)\lambda_i +\lambda_{k+1}-\lambda)\leq 0.
\]
Clearly, $\lambda_{k+1}-\lambda_i$ is non-increasing in $i$.
Put $f(x):=x^2-(\lambda_{k+1}-2\lambda_1+5)x$. Then, the function $f$ is non-increasing in the interval $(-\infty,(\lambda_{k+1}-2\lambda_1+5)/2]$. From Lemma \ref{lem: upperbound}, $(\lambda_{k+1}-2\lambda_1+5)/2\geq 2$.
Since $0\leq \lambda_i \leq 2$, $\lambda_i^2-(\lambda_{k+1}-2\lambda_1+5)\lambda_i +\lambda_{k+1}-\lambda$ is non-increasing in $i$. We may use Lemma \ref{Chebyineq} and thus
\[
\left(\lambda_{k+1}-\sum_{i=0}^k\frac{\lambda_i}{k+1}\right)\left(\sum_{i=0}^k\frac{(1-\lambda_i)\lambda_{k+1}+\lambda_i^2 -(5-2\lambda_1)\lambda_i}{k+1}-\lambda_1\right)\leq 0.
\]
If $k\geq m(\lambda_1)$, then $\lambda_{k+1}-\sum_{i=0}^k\lambda_i/(k+1)$ is strictly positive. In this case, we have
\[
\frac{1}{k+1}\sum_{i=0}^k((1-\lambda_i)\lambda_{k+1}+\lambda_i^2 -(5-2\lambda_1)\lambda_i-\lambda_1)\leq 0.
\]
By Lemma \ref{lem: admatrix}, we obtain
\[
\lambda_{k+1}\leq \frac{(k+1)\lambda_1+\sum_{i=1}^k((5-2\lambda_1)\lambda_i-\lambda_i^2)}{\sum_{i=0}^k(1-\lambda_i)}.
\]
This inequality also holds for $k<m(\lambda_1)$.
\end{proof}
\begin{proof}[Proof of Corollary 1.3]
If $k=m(\lambda_1)$, then $\lambda_{k+1}=\mu_2$ and $\lambda_1=\cdots=\lambda_{k-1}=\mu_1$. By Theorem \ref{main thm 2},  we have
\[
\frac{\mu_2}{\mu_1}\leq \frac{6m+1-3m\mu_1}{m+1-m\mu_1}.
\]
Let $g(x):=(6m+1-3mx)/(m+1-mx)$. The function $g$ is increasing. By Lemma \ref{lem: upperbound},
\[
\frac{\mu_2}{\mu_1}\leq g(1)=3m+1.
\]
\end{proof}

\section{On the non-triviality of Corollary 1.3}
In this section, we consider symmetric graphs, other than complete graphs.
Let $\mu_1$ and $\mu_2$ be the first and the second smallest positive eigenvalue, respectively.
If $(3m+1)\mu_1$ is not less than $2$, then the inequality in Corollary \ref{main cor} is trivial since $\mu_2\leq 2$ always holds.
In this section, we see that there exist infinitely many graphs such that $(3m+1)\mu_1$ is strictly less than $2$.

Let $C_{N}$, $N\geq 3$, be the cycle graph with $N$ vertices.
Cycle graphs are symmetric.
The spectra of cycle graphs are well-known. 
\begin{lem}
The smallest positive eigenvalue of the normalized Laplace operator associated with $C_{N}$ is
$1-\cos (2\pi/N)$ and its multiplicity is $2$.
\end{lem}
Since $1-\cos (2\pi/N)$ is decreasing in $N$ and tends to $0$ as $N\to\infty$,
there exists a number $N_0$ such that $(3m+1)\mu_1=7(1-\cos (2\pi/N))$ is strictly less than $2$ for any $N\geq N_0$. In fact, we can take $N_0=9$.
\bibliographystyle{plain}
\bibliography{reference}

\begin{bibdiv}
\begin{biblist}

\bib{MR2115463}{article}{
      author={Cheng, Qing-Ming},
      author={Yang, Hongcang},
       title={Estimates on eigenvalues of {L}aplacian},
        date={2005},
        ISSN={0025-5831},
     journal={Math. Ann.},
      volume={331},
      number={2},
       pages={445\ndash 460},
         url={https://doi.org/10.1007/s00208-004-0589-z},
}

\bib{MR578204}{article}{
      author={Hile, G.~N.},
      author={Protter, M.~H.},
       title={Inequalities for eigenvalues of the {L}aplacian},
        date={1980},
        ISSN={0022-2518},
     journal={Indiana Univ. Math. J.},
      volume={29},
      number={4},
       pages={523\ndash 538},
         url={https://doi.org/10.1512/iumj.1980.29.29040},
}

\bib{HLS}{article}{
      author={Hua, B.},
      author={Lin, Y.},
      author={Su, Y.},
       title={{P}ayne-{P}\'{o}lya-{W}einberger, {H}ile-{P}rotter and {Y}ang's
  inequalities for {D}irichlet {L}aplace eigenvalues on integer lattices},
        date={2017},
     journal={arXiv:1710.05799},
}

\bib{MR84696}{article}{
      author={Payne, L.~E.},
      author={P\'{o}lya, G.},
      author={Weinberger, H.~F.},
       title={On the ratio of consecutive eigenvalues},
        date={1956},
        ISSN={0097-1421},
     journal={J. Math. and Phys.},
      volume={35},
       pages={289\ndash 298},
         url={https://doi.org/10.1002/sapm1956351289},
}

\bib{MR73046}{article}{
      author={Payne, Laurent~E.},
      author={P\'{o}lya, Georges},
      author={Weinberger, Hans~F.},
       title={Sur le quotient de deux fr\'{e}quences propres cons\'{e}cutives},
        date={1955},
        ISSN={0001-4036},
     journal={C. R. Acad. Sci. Paris},
      volume={241},
       pages={917\ndash 919},
}

\bib{MR257592}{article}{
      author={Thompson, Colin~J.},
       title={On the ratio of consecutive eigenvalues in {$N$}-dimensions},
        date={1969},
        ISSN={0022-2526},
     journal={Studies in Appl. Math.},
      volume={48},
       pages={281\ndash 283},
         url={https://doi.org/10.1002/sapm1969483281},
}

\bib{Yang}{article}{
      author={Yang, Hongcang},
       title={An estimate of the difference between consecutive eigenvalues},
        date={1995},
     journal={Preprint IC/91/60 of ICTP ,Trieste},
}

\end{biblist}
\end{bibdiv}
\end{document}